\documentclass[11pt,a4paper]{amsart}

\usepackage[T1]{fontenc}
\usepackage[utf8]{inputenc}

\usepackage{amssymb}

\makeatletter
\@namedef{subjclassname@2010}{%
  \textup{2010} Mathematics Subject Classification}
\makeatother

\allowdisplaybreaks[2]

\newtheorem{theorem}{Theorem}[section]
\newtheorem{proposition}[theorem]{Proposition}
\newtheorem{lemma}[theorem]{Lemma}
\newtheorem{corollary}[theorem]{Corollary}
\theoremstyle{definition}

\newtheorem{example}[theorem]{Example}
\newtheorem{question}[theorem]{Question}
\numberwithin{equation}{section}

\begin{document}

\title{Zero Jordan product determined Banach algebras}

\author{J. Alaminos}
\address{Departamento de An\' alisis
    Matem\' atico\\ Fa\-cul\-tad de Ciencias\\ Universidad de Granada\\
    18071 Granada, Spain}
\email{alaminos@ugr.es}
\author{M. Bre\v sar}
\address{Faculty of Mathematics and Physics\\
    University of Ljubljana\\
    Jadranska 19, 1000 Ljubljana,
    and \\
    Faculty of Natural Sciences and Mathematics\\
    University of Maribor\\
    Koro\v ska 160, 2000 Maribor, Slovenia} \email{matej.bresar@fmf.uni-lj.si}
\author{J. Extremera}
\address{Departamento de An\' alisis
    Matem\' atico\\ Fa\-cul\-tad de Ciencias\\ Universidad de Granada\\
    18071 Granada, Spain}
\email{jlizana@ugr.es}
\author{A.\,R. Villena}
\address{Departamento de An\'alisis
    Matem\' atico\\ Fa\-cul\-tad de Ciencias\\ Universidad de Granada\\
    18071 Granada, Spain}
\email{avillena@ugr.es}

\begin{abstract}
A  Banach algebra $A$ is said to be a zero Jordan product determined
Banach algebra if every continuous bilinear map $\varphi\colon A\times A\to X$,
where $X$ is an arbitrary Banach space, which satisfies $\varphi(a,b)=0$ whenever
$a$, $b\in A$ are such that $ab+ba=0$, is of the form $\varphi(a,b)=\sigma(ab+ba)$
for some continuous linear map $\sigma$. We show that all $C^*$-algebras and all
group algebras $L^1(G)$ of amenable locally compact groups have this
property, and also discuss some applications.
\end{abstract}

\thanks{
The authors were supported by MINECO grant MTM2015--65020--P.
The first, the third  and the fourth named authors were supported
by  Junta de Andaluc\'{\i}a grant FQM--185.
The second  named author was supported by ARRS grant P1--0288.}

\maketitle

\section{Introduction}

The purpose of this paper is to show that some fundamental results on zero product determined Banach algebras and 
zero Lie product determined Banach algebras also hold for zero Jordan product determined Banach algebras.
In the next paragraphs, we give definitions of these and  related notions, and recall the relevant results.

Let $A$ be a Banach algebra. For $a,b\in A$, we write $a\circ b=ab+ba$
and $[a,b]=ab-ba$.
We denote by $A^2$, $A\circ A$, and $[A,A]$ the linear span of all elements of the form
$ab$, $a\circ b$, and $[a,b]$ ($a,b\in A$), respectively.
We say that $A$ is a \emph{zero product determined Banach algebra} if every continuous bilinear map
$\varphi\colon A\times A\to X$, where $X$ is an arbitrary Banach space, with the property that
\begin{equation}\label{P}
a,b\in A, \ ab=0 \ \Longrightarrow \ \varphi(a,b)=0
\end{equation}
can be written in the standard form
\begin{equation}\label{PS}
\varphi(a,b)=\sigma(a b) \quad (a,b\in A)
\end{equation}
for some continuous linear map $\sigma\colon A^2\to X$.
This concept appeared as a byproduct of the so-called property $\mathbb{B}$ introduced in \cite{ABEV0}.
We say that $A$ has \emph{property $\mathbb{B}$} if for every continuous bilinear map
$\varphi \colon A\times A\to X$, where $X$ is an arbitrary Banach space, the
condition~\eqref{P} implies the condition
\begin{equation}\label{B11}
\varphi(ab,c)=\varphi(a,bc)  \quad (a,b,c\in A).
\end{equation}
In~\cite{ABEV0} it was shown that many important examples of Banach algebras,
including $C^*$-algebras and group algebras $L^1(G)$, where $G$ is any locally compact group, have property $\mathbb{B}$.
Suppose that the Banach algebra $A$ has the property $\mathbb{B}$ and has a bounded approximate identity,
i.e., a bounded net $(e_\lambda)_{\lambda\in\Lambda}$ such that
$\lim_{\lambda\in\Lambda}e_\lambda a=\lim_{\lambda\in\Lambda}a e_\lambda=a$ for each $a\in A$.
If $X$ is a Banach space and $\varphi\colon A\times A\to X$ is a continuous bilinear map that
satisfies~\eqref{P}, then~\eqref{B11} holds. Setting $c=e_\lambda$ ($\lambda\in\Lambda$) we
see that the net $(\varphi(a,e_\lambda))_{\lambda\in\Lambda}$ converges for each $a\in A^2$,
and defining $\sigma(a)=\lim_{\lambda\in\Lambda}\varphi(a,e_\lambda)$ we obtain a continuous
linear map that satisfies~\eqref{PS}
(it is worth pointing out that, in this case, $A^2=A$ because of the factorization theorem).
Consequently, $A$ is a zero product determined Banach algebra.
This remark applies to $C^*$-algebras and group algebras, and, therefore, all of them are zero product determined Banach algebras.

We can define Lie and Jordan versions of the zero product determination in a natural way.
We say that $A$ is a \emph{zero Lie product determined Banach  algebra} if every continuous bilinear map
$\varphi\colon A\times A\to X$, where $X$ is an arbitrary Banach space, with the property that
\begin{equation*}
a,b\in A, \ [a,b]=0 \ \Longrightarrow \ \varphi(a,b)=0
\end{equation*}
can be written in the standard form
\begin{equation*}
\varphi(a,b)=\sigma([a, b]) \quad (a,b\in A)
\end{equation*}
for some continuous linear map $\sigma\colon [A,A]\to X$.
This notion has been introduced in our recent papers \cite{ABEVLie0,ABEVLie},
where we have shown that
every weakly amenable Banach algebra with property $\mathbb{B}$ and having a
bounded approximate identity is a zero Lie product determined Banach algebra.
Consequently, all $C^*$-algebras and all group algebras are zero Lie product determined Banach algebras.

Finally,
we say that $A$ is a \emph{zero Jordan product determined Banach algebra} if every continuous bilinear map
$\varphi\colon A\times A\to X$, where $X$ is an arbitrary Banach space, with the property that
\begin{equation}\label{J}
a,b\in A, \ a\circ b=0 \ \Longrightarrow \ \varphi(a,b)=0
\end{equation}
can be written in the standard form
\begin{equation}\label{S}
\varphi(a,b)=\sigma(a\circ b) \quad (a,b\in A)
\end{equation}
for some continuous linear map $\sigma\colon A\circ A\to X$.
The main goal of this paper is to show that every $C^*$-algebra, as well as every group algebra $L^1(G)$ of an amenable locally compact group $G$,
is a zero Jordan product determined Banach algebra. Applications are also discussed, in particular those about the representation of commutators. Some remarks concerning the purely algebraic theory of zero Jordan product determined  algebras are also given at the end.

Let $X$ be a Banach space. We write $X^*$ for the dual of $X$. If  $S\subset X$, then
$\operatorname{span}\,S$ and  $\overline{\operatorname{span}}\,S$ 
stand for the linear span and the closed linear span of  $S$, respectively.

\section{Alternative definition}

We next show that the role of the Banach space $X$ in the definition of a 
zero Jordan product determined Banach algebra is ancillary, and can be replaced by $\mathbb{C}$.

\begin{proposition}\label{1046}
Let $A$ be a Banach algebra. Then the following properties are equivalent:
\begin{enumerate}
\item 
$A$ is a zero Jordan product determined Banach algebra,
\item
every continuous bilinear functional $\varphi\colon A\times A\to\mathbb{C}$ that satisfies~\eqref{J}
is of the form~\eqref{S} for some continuous linear functional $\sigma\in A^*$.
\end{enumerate}
\end{proposition}

\begin{proof}
Suppose that (1) holds.
Let $\varphi\colon A\times A\to\mathbb{C}$ be a continuous bilinear functional
satisfying~\eqref{J}. By applying property (2) with $X=\mathbb{C}$, we get  $\tau\in (A\circ A)^*$ such that
$\varphi(a,b)=\tau(a\circ b)$ $(a,b\in A)$. The functional
$\tau$ can be extended to a continuous linear functional on $A$ so that (1) is obtained.

We now assume that (2) holds.
Let $X$ be a Banach space, and let $\varphi\colon A\times A\to X$ be a continuous bilinear map
that satisfies \eqref{J}.
For each $\xi\in X^*$, the continuous bilinear functional $\xi\circ\varphi\colon A\times A\to\mathbb{C}$
satisfies~\eqref{J}. Therefore there exists a unique $\tau(\xi)\in (A\circ A)^{*}$ such that
$\xi(\varphi(a,b))=\tau(\xi)(a\circ b)$ for all $a,b\in A$.
It is clear that the map $\tau\colon X^*\to (A\circ A)^*$
is linear. We next show that $\tau$ is continuous. Let $(\xi_n)$ be a sequence in $X^*$ with $\lim \xi_n=0$
and $\lim \tau(\xi_n)=\xi$ for some $\xi\in (A\circ A)^*$. For each $a,b\in A$, we have
\[
0=\lim\xi_n(\varphi(a,b))=\lim\tau(\xi_n)(a\circ b)=\xi(a\circ b).
\]
We thus have $\xi=0$, and the closed graph theorem yields the continuity of $\tau$.

For all $a_1,\ldots,a_n,b_1,\ldots,b_n\in A$ and $\xi\in X^*$ we have
\begin{equation}\label{11667}
\begin{split}\xi\Bigl(\sum_{k=1}^n\varphi(a_k,b_k)\Bigr) & =
\sum_{k=1}^n\xi\bigl(\varphi(a_k,b_k)\bigr) =\sum_{k=1}^n\tau(\xi)(a_k\circ b_k)
\\
& = \tau(\xi)\Bigl(\sum_{k=1}^n a_k\circ b_k\Bigr).
\end{split}
\end{equation}
Consequently, if $a_1,\ldots,a_n,b_1,\ldots,b_n\in A$ are such that $\sum_{k=1}^n a_k\circ b_k=0$,
then $\xi\bigl(\sum_{k=1}^n\varphi(a_k,b_k)\bigr)=0$ for each $\xi\in X^*$, and hence
$\sum_{k=1}^n\varphi(a_k,b_k)=0$. We thus can define a linear map $\sigma\colon A\circ A\to X$ by
\[
\sigma\Bigl(\sum_{k=1}^n a_k\circ b_k\Bigr)=\sum_{k=1}^n\varphi(a_k,b_k)
\]
for all $a_1,\ldots,a_n,b_1,\ldots,b_n\in A$. Of  course, $\varphi(a,b)=\sigma(a\circ b)$ for all $a,b\in A$.
Our next concern is the continuity of $\sigma$.
Let  $a_1,\ldots,a_n,b_1,\ldots,b_n\in A$. Then there exists $\xi\in X^*$ with $\Vert  \xi \Vert =1$ such that
\[
\xi\Bigl(\sum_{k=1}^n\varphi(a_k,b_k)\Bigr)=\Bigl\Vert\sum_{k=1}^n\varphi(a_k,b_k)\Bigr\Vert.
\]
On account of~\eqref{11667}, we have
\begin{align*}
\Bigl\Vert \sigma\Bigl(\sum_{k=1}^n a_k\circ b_k\Bigr)\Bigr\Vert & =
\Bigl\Vert\sum_{k=1}^n\varphi(a_k,b_k)\Bigr\Vert  = \xi\Bigl(\sum_{k=1}^n\varphi(a_k,b_k)\Bigr) \\
& = \Bigl\vert\tau(\xi)\Bigl(\sum_{k=1}^n a_k\circ b_k\Bigr)\Bigr\vert
 \le \Vert\tau(\xi)\Vert\Bigl\Vert\sum_{k=1}^n a_k\circ b_k\Bigr\Vert \\
& \le \Vert\tau\Vert\Bigl\Vert\sum_{k=1}^n a_k\circ b_k\Bigr\Vert,
\end{align*}
which proves the continuity of $\sigma$, and hence that property (1) holds.
\end{proof}

\section{Amenability-like properties}

Recall that a Banach algebra $A$ is said to be \emph{amenable} if
every continuous derivation from $A$ into $X^*$ is inner,
whenever $X$ is a Banach $A$-bimodule. It is known (see \cite[Theorem 2.9.65]{D})
that the Banach algebra $A$ is amenable if and only if it has an
{\em approximate diagonal}, that is, a bounded net
$(\mathbf{t}_\lambda)_{\lambda\in\Lambda}$ in the projective tensor product
$A\widehat{\otimes}A$ with the properties
\begin{equation}\label{E1}
\lim_{\lambda\in\Lambda}\, \left(a\cdot\mathbf{t}_\lambda-\mathbf{t}_\lambda\cdot a\right)=0
\end{equation}
and
\begin{equation}\label{E2}
\lim_{\lambda\in\Lambda}\pi(\mathbf{t}_\lambda)a=a
\end{equation}
for each $a\in A$, where the operations on $A\widehat{\otimes}A$
are defined, for simple tensors, by
\[
a\cdot(b\otimes c)=ab\otimes c, \
(b\otimes c)\cdot a=b\otimes ca, \ \text{and} \
\pi(b\otimes c)=bc
\]
for all $a,b,c\in A$. The flip map on $A\widehat{\otimes}A$ is defined by
\[
(a\otimes b)^\circ=b\otimes a\qquad(a,b\in A),
\]
and a tensor $\mathbf{t}$ of $A\widehat{\otimes}A$ is called \emph{symmetric} if
$\mathbf{t}^\circ=\mathbf{t}$.
A Banach algebra $A$ is said to be \emph{symmetrically amenable} if it has an
approximate diagonal consisting of symmetric tensors.
This concept was introduced by B. E. Johnson in \cite{J2}, where properties and
examples of symmetrically amenable Banach algebras can be found.
Most, but not all, amenable Banach algebras are symmetrically amenable.
The group algebra $L^1(G)$ is symmetrically amenable for each locally compact amenable  group $G$ (\cite[Theorem 4.1]{J2}).

A Banach algebra $A$ is said to be \emph{weakly amenable}
if every continuous derivation from $A$ into $A^*$ is inner.
For a thorough treatment of this property and an account of many interesting examples of
weakly amenable Banach algebras we refer the reader to~\cite{D}.
We should remark that each $C^*$-algebra and the group algebra $L^1(G)$ of each locally compact group
$G$ are weakly amenable~\cite[Theorems 5.6.48 and 5.6.77]{D}.
If the Banach algebra $A$ is commutative, then $A^*$ is a commutative Banach $A$-bimodule and therefore
a derivation from $A$ into $A^*$ is inner if and only if it is zero.
It is worth pointing out that a basic obstruction to the weak amenability
is the existence of non-zero, continuous point derivations~\cite[Theorem~2.8.63(ii)]{D}. Recall that a linear functional $D$ on $A$
is a \emph{point derivation at} a given multiplicative linear functional $\vartheta$ if
\begin{equation}\label{pd}
D(ab)=D(a)\vartheta(b)+\vartheta(a)D(b) \quad  (a,b\in A).
\end{equation}

Weak amenability and zero Jordan and Lie product determination are related through the following result.

\begin{theorem}\label{t1}\cite{ABEVLie}
Let $A$ be a  weakly amenable Banach algebra with property $\mathbb{B}$ and having a bounded approximate identity.
If a continuous bilinear functional $\varphi \colon A\times A\to\mathbb{C}$ satisfies
\begin{equation*}
a,b\in A, \ ab=ba=0 \ \implies \ \varphi(a,b)=0,
\end{equation*}
then  there exist $\sigma,\tau\in A^*$ such that
\[
\varphi(a,b)= \sigma(a\circ b) + \tau([a,b])\quad (a,b\in A).
\]
\end{theorem}

\begin{lemma}\label{1401}
Let $A$ be a  weakly amenable Banach algebra with property $\mathbb{B}$ and having a bounded approximate identity.
If a continuous bilinear functional $\varphi \colon A\times A\to\mathbb{C}$ satisfies
\begin{equation*}
a,b\in A, \ a\circ b=0 \ \Longrightarrow \ \varphi(a,b)=0,
\end{equation*}
then  there exist $\sigma,\tau\in A^*$  with $\tau\left( \bigl[[A,A],[A,A]\bigr]\right)=\{0\}$ such that
\[
\varphi(a,b)= \sigma(a\circ b)+\tau([a,b])\quad (a,b\in A).
\]
\end{lemma}

\begin{proof}
If $a,b\in A$ are such that $ab=ba=0$, then $a\circ b=0$ and so $\varphi(a,b)=0$.
Therefore $\varphi$ satisfies the assumption in Theorem~\ref{t1}. Hence
there exist $\sigma,\tau\in A^*$ such that
\begin{equation}\label{1906}
\varphi(a,b)= \sigma(a\circ b) + \tau([a,b])\quad (a,b\in A).
\end{equation}
Suppose that $a,b\in A$ are such that $a^2=b^2$. Then
\[
(a-b)\circ(a+b)=0,
\]
and \eqref{1906} yields
\[
0=\varphi(a-b,a+b)=\tau([a-b,a+b])=2\tau([a,b]).
\]
In particular, for all square-zero elements $a,b\in A$ we have
\[
\tau([a,b])=0.
\]

By \cite[Theorem 2.1]{ABESV}, every commutator in A
lies in the closed linear span of square-zero elements,
and therefore $\tau\left(\bigl[[A,A],[A,A]\bigr]\right)=\{0\}$.
\end{proof}

\begin{theorem}\label{1400}
Let $A$ be a symmetrically amenable Banach algebra. Then
\[
[a,A]\subset
\overline{\bigl[a,[A,A]\bigr]}
\]
for each $a\in A$.
Consequently,
\[
\overline{[A,A]}=\overline{\bigl[[A,A],[A,A]\bigr]}.
\]
\end{theorem}

\begin{proof}
We first prove that
\begin{equation}\label{E3}
\pi(\mathbf{t})ab-b\pi(\mathbf{t})a+
\pi\bigl((b\cdot\mathbf{t}^\circ-\mathbf{t}^\circ\cdot b)\star a\bigr)\in \bigl[b,[A,A]\bigr]
\end{equation}
for all $a,b\in A$ and $\mathbf{t}\in A\widehat{\otimes}A$,
where $\star$ is the operation on $A\widehat{\otimes}A$ defined, for simple tensors, by
\[
(u\otimes v)\star a=(ua)\otimes v \quad (u,v,a\in A).
\]
It suffices to check the formula for elementary tensors $\mathbf{t}=u\otimes v$ with $u,v\in A$.
We have
\begin{align*}
\pi(\mathbf{t})ab-b\pi(\mathbf{t})a+
\pi\bigl((b\cdot\mathbf{t}^\circ-\mathbf{t}^\circ\cdot b)\star a\bigr) & =
uvab-buva \\
& \quad {}+ \pi\bigl((bva)\otimes u-(va)\otimes (ub)\bigr) \\
& = uvab-buva+bvau-vaub\\
& =\bigl[[u,va],b\bigr].
\end{align*}

Let $(\mathbf{t}_\lambda)_{\lambda\in\Lambda}$ be an approximate diagonal for
$A$ consisting of symmetric tensors. From~\eqref{E1} and~\eqref{E2}
we deduce that
\begin{equation}\label{E4}
\lim_{\lambda\in\Lambda}
\Bigl(
\pi(\mathbf{t}_\lambda)ab-b\pi(\mathbf{t}_\lambda)a+
\pi\bigl((b\cdot\mathbf{t}_\lambda-\mathbf{t}_\lambda\cdot b)\star a\bigr)
\Bigr)=[a,b].
\end{equation}
Finally, \eqref{E3} and \eqref{E4} yield $[a,b]\in\overline{\bigl[b,[A,A]\bigr]}$.

In order to prove the last assertion,
set $a,b\in A$, and let $\varepsilon\in\mathbb{R}^+$.
Then there exists $u\in [A,A]$ such that
\[
\left\Vert [a,b]-[a,u] \right\Vert<\varepsilon/2
\]
and there exists $v\in [A,A]$ such that
\[
\Vert [u,a]-[u,v]\Vert<\varepsilon/2.
\]
Then $[u,v]\in\bigl[[A,A],[A,A]\bigr]$ and
\[
\Vert [a,b]-[v,u]\Vert\le
\Vert [a,b]-[a,u]\Vert+\Vert[u,a]-[u,v]\Vert<\varepsilon. \qedhere
\]
\end{proof}

Since the group algebra $L^1(G)$ is symmetrically amenable for each locally compact amenable  group $G$ (\cite[Theorem 4.1]{J2}),
the following result is an obvious consequence of Theorem~\ref{1400}.

\begin{corollary}\label{1550}
Let $G$ be an amenable locally compact group.
Then
\[
[f,L^1(G)]\subset
\overline{\bigl[f,[L^1(G),L^1(G)]\bigr]}
\]
for each $f\in L^1(G)$.
Consequently,
\[
\overline{[L^1(G),L^1(G)]}=\overline{\bigl[[L^1(G),L^1(G)],[L^1(G),L^1(G)]\bigr]}.
\]
\end{corollary}

\begin{corollary}\label{1552}
$A$ be a $C^*$-algebra. Then $\overline{[A,A]}=\overline{\bigl[[A,A],[A,A]\bigr]}$.
\end{corollary}

\begin{proof}
We first assume that $A$ is a von Neumann algebra.
Let $p,q\in A$ be two projections, and take $u=2p-1, v=2q-1\in A$. 
Then $u$ and $v$ are unitary, and we write $G$ for the subgroup of the unitary group of $A$ generated by $u$ and $v$.
Let $D_\infty$ be the infinite dihedral group; this is the group generated
by two elements $s$ and $t$ with the relations $s^2=e$ and $t^2=e$.
Since $u^2=1$ and $v^2=1$, we conclude that there exists a
homomorphism from $D_\infty$ onto $G$.
Since $D_\infty$ is solvable, it follows that $G$ is also solvable, and hence amenable (\cite[Proposition 3.3.61]{D});
here, $G$ is equipped with the discrete topology.
Since $G$ consists of unitary elements, we see that $\Vert w\Vert =1$ for each $w\in G$, and therefore, 
for each $a\in\ell^1(G)$, we have
\[
\sum_{w\in G}\Vert a(w)w\Vert=\sum_{w\in G}\vert a(w)\vert=\Vert a\Vert_{\ell^1(G)}.
\]
Consequently, we can define a continuous homomorphism $\Phi\colon\ell^1(G)\to A$ by
\[
\Phi(a)=\sum_{w\in G}a(w)w\quad (a\in\ell^1(G)).
\]
Further, $\Phi(\delta_1)=1$, $\Phi(\delta_u)=u$, and $\Phi(\delta_v)=v$,
where $\delta_w$ stands for the characteristic function of the point $w\in G$.
From Corollary \ref{1550} we now deduce that
\[
[\delta_u+\delta_1,\delta_v+\delta_1]\in
\overline{\bigl[[\ell^1(G),\ell^1(G)],[\ell^1(G),\ell^1(G)]\bigr]},
\]
whence
\begin{equation*}
\begin{split}
[p,q]
&=
\Big[\Phi\Big(\frac{\delta_u+\delta_1}{2}\Big),\Phi\Big(\frac{\delta_v+\delta_1}{2}\Big)\Big]\\
&=\frac{1}{4}
\Phi\bigl([\delta_u+\delta_1,\delta_v+\delta_1]\bigr)\\
&\in
\Phi\left(\overline{\bigl[[\ell^1(G),\ell^1(G)],[\ell^1(G),\ell^1(G)]\bigr]}\right)\\
&\subset
\overline{\Phi\bigl(\bigl[[\ell^1(G),\ell^1(G)],[\ell^1(G),\ell^1(G)]\bigr]\bigr)}\\
&\subset
\overline{\bigl[[A,A],[A,A]\bigr]}.
\end{split}
\end{equation*}
Since $A$ is the closed linear span of its idempotents, it follows that
\[
[a,b]\in\overline{\bigl[[A,A],[A,A]\bigr]}
\]
for all $a,b\in A$.

We now consider the case when $A$ is an arbitrary $C^*$-algebra, and suppose that
the result does not hold. Then there exist $a,b\in A$ and $\sigma\in A^*$ such that
\begin{equation}\label{182}
\sigma\bigl( [a,b]\bigr)\ne 0
\end{equation}
and
\begin{equation}\label{1823}
\sigma\bigl(\bigl[[A,A],[A,A]\bigr]\bigr)=0.
\end{equation}
We then consider the von Neumann algebra $A^{**}$ and elements $x,y,z,w\in A^{**}$.
There exists nets $(a_i)$, $(b_j)$, $(c_k)$, and $(d_l)$ such that
\[
\lim a_i=x, \quad
\lim b_j=y, \quad
\lim c_k=z, \quad
\lim d_l=w,
\]
with respect to the weak$^*$ topology on $A^{**}$. By~\eqref{1823},
\[
\sigma
\Bigl(\bigl[[a_i,b_j],[c_k,d_l]\bigr]\Bigr)=0  \quad (i,j,k,l).
\]
Taking limits in the preceding equation, and using the separate continuity with respect to the weak$^*$ topology
of the product in $A^{**}$, we see that
\begin{align*}
0 & =\lim_i\lim_j\lim_k\lim_l
\sigma\bigl(\bigl[[a_i,b_j],[c_k,d_l]\bigr]\bigr) \\
& = \lim_i\lim_j\lim_k
\sigma\bigl(\bigl[[a_i,b_j],[c_k,w]\bigr]\bigr) \\
& = \lim_i\lim_j
\sigma\bigl(\bigl[[a_i,b_j],[z,w]\bigr]\bigr) \\
& = \lim_i
\sigma\bigl(\bigl[[a_i,y],[z,w]\bigr]\bigr) \\
& = \sigma\bigl(\bigl[[x,y],[z,w]\bigr]\bigr).
\end{align*}
Since $A^{**}$ is a von Neumann algebra, it follows that
\[
[a,b]\in\overline{\bigl[[A^{**},A^{**}],[A^{**},A^{**}]\bigr]}
\]
(the closure being taken with respect to the norm topology) and therefore that
$\sigma\bigl( [a,b]\bigr)=0$, which contradicts~\eqref{182}.
\end{proof}

From Proposition \ref{1046}, Lemma \ref{1401}, and Corollaries \ref{1550} and \ref{1552} we obtain the following.

\begin{theorem}
The following Banach algebras are zero Jordan product determined:
\begin{enumerate}
\item
the group algebra $L^1(G)$ of an amenable locally compact group;
\item
every $C^*$-algebra.
\end{enumerate}
\end{theorem}

\begin{question}
Is the group algebra $L^1(G)$ a zero Jordan product determined Banach algebra for each locally compact group $G$?
\end{question}

We will now provide an example of  a noncommutative Banach algebra that is not zero Jordan product determined.
As a matter of fact, we will only slightly modify the discussion from \cite{ABEVLie} on examples of Banach algebras that are not zero product determined. Recall that a Banach algebra $A$ is said to be {\em essential} if $\overline{A^2} = A$.

\begin{lemma}
If a Banach algebra $A$ is essential and zero Jordan product determined,
then there are no non-zero, continuous point derivations on $A$.
\end{lemma}

\begin{proof}
Suppose there exists a non-zero, continuous point derivation $D$ on $A$ at a multiplicative functional $\vartheta$. Since $A$ is essential, $\vartheta\ne 0$.
Define a continuous bilinear functional $\varphi\colon A\times A\to\mathbb{C}$ by
\[
\varphi(a,b)=\vartheta(a)D(b) \quad  (a,b\in A).
\]
Take $a,b\in A$ such that $a\circ b=0$. Our goal is to show that  $\varphi(a,b)=0$. Obviously, $a\circ b=0$ implies $\vartheta(a)\vartheta(b)=0$ and hence either $\vartheta(a)=0$ or $\vartheta(b)=0$.
If $\vartheta(a)=0$, then $\varphi(a,b)=0$. We may thus assume that $\vartheta(b)=0$.
But then
\[
0=D(ab+ba)=\vartheta(a)D(b) + D(b)\vartheta(a)= 2 \varphi(a,b),
\]
so $\varphi(a,b)=0$ in this case too. Since $A$ is zero Jordan product determined, $\varphi$ is, in particular, symmetric. That is,
\[
\vartheta(a)D(b) = \vartheta(b)D(a)\quad  (a,b\in A).
\]
Since $\vartheta\ne 0$, this implies that $D$ is scalar multiple of $\vartheta$.
However, this is clearly impossible in light of~\eqref{pd}.
\end{proof}

Referring to the proofs of \cite[Proposition~2.3, Corollary~2.4]{ABEVLie}, 
we have that  the Banach space $X$ constructed by Read  \cite{R} has the 
property that the Banach algebra $\mathcal{B}(X)$ of all bounded linear 
operators on $X$ has a non-zero, continuous point derivation. Hence, 
the following is true.

\begin{proposition}
There exists a Banach space $X$ such that $\mathcal{B}(X)$ is not 
zero Jordan product determined.
\end{proposition}

We remark that the most standard Banach spaces $X$ are isomorphic 
to $X\oplus X$, and so $\mathcal{B}(X)$ is isomorphic to the matrix 
algebra $M_2(\mathcal{B}(X))$. It is easy to see that  every $n\times n$ 
matrix algebra (with $n\ge 2$)  over any unital algebra is generated 
by its idempotents (see~\cite[Proposition~4.2]{ABESV}). Such an algebra 
is therefore, by Theorem~\ref{idem} below, zero Jordan product determined even 
in the algebraic sense.

\section{Applications}

We start with a characterization of elements from the center $\mathcal{Z}(A)$ of $A$.

\begin{proposition}
Let $A$ be a semiprime zero Jordan product determined Banach algebra.  Then the following 
properties are equivalent for an element $a\in A$:
\begin{enumerate}
\item
$a\in\mathcal{Z}(A)$,
\item
$(a\circ x)\circ y=0$ whenever $x,y\in A$ are such that $x\circ y=0$.
\end{enumerate}
\end{proposition}

\begin{proof}
It suffices to show that (2) implies (1).
Let $\varphi\colon A\times A\to A$ be the continuous bilinear map defined by
$\varphi(x,y)=(a\circ x)\circ y$ $(x,y\in A)$.
Then there exists a continuous linear map $\sigma\colon A\circ A\to A$
such that $\varphi(x,y)=\sigma(x\circ y)$ $(x,y\in A)$, which, in 
particular, implies that $\varphi$ is symmetric. We thus get
\begin{equation}\label{e1457}
\bigl[a,[x,y]\bigr]=
(a\circ x)\circ y-(a\circ y)\circ x=
\varphi(x,y)-\varphi(y,x)=
0
\end{equation}
for all $x, y\in A$.
Let $D$ be the inner derivation of $A$ implemented by $a$.
From \eqref{e1457} we deduce that $D^2=0$. It is well-known that, in every semiprime algebra, this implies $D=0$. Let us give the proof for the sake of completeness. Using the derivation law, we  obtain
\[
2D(x)D(y)=D^2(xy) - D^2(x)y - xD^2(y) =0,
\]
and hence
\[
D(x)yD(x)= D(x)D(yx) - D(x)D(y)x=0
\]
for all $x,y\in A$.
Since $A$ is semiprime, it may be concluded that
$D=0$, meaning that  $[a,x]=0$ for each $x\in A$.
\end{proof}

If  elements $a$ and $b$ of a Banach algebra $A$  satisfy $a\circ b=0$, then $ab$ is a commutator, namely $ab= \frac{1}{2}[a,b]$.
The next proposition shows that in a noncommutative zero Jordan product determined Banach algebra, every commutator lies
in the closed linear span of such elements.

\begin{theorem}\label{pr2}
Let $A$ be a zero Jordan product determined Banach algebra.
Then every commutator in $A$ lies in $\overline{\operatorname{span}}\,\{ab :  a,b\in A, a\circ b =0\}$.
\end{theorem}

\begin{proof}
Let $J=\overline{\operatorname{span}}\,\{ab : a,b\in A, a\circ b =0\}$.
Define $\varphi\colon A\times A\to A/J$ by
\[
\varphi(x,y)= xy + J \quad (x,y\in A).
\]
Then $\varphi$ is a continuous bilinear map and it is clear that $x\circ y=0$ implies $\varphi(x,y)=0$.
Therefore, $\varphi$ must be, in particular, symmetric, which readily implies that $[x,y]\in J$ for all $x,y\in A$.
\end{proof}

\begin{corollary}
Let $A$ be a  weakly amenable Banach algebra with property $\mathbb{B}$ and having a bounded approximate identity.
Then $A$ is a zero Jordan product determined Banach algebra if and only if
 every commutator in $A$ lies in $\overline{\operatorname{span}}\,\{ab : a,b\in A, a\circ b =0\}$.
\end{corollary}

\begin{proof}
By Theorem \ref{pr2}, we only have to prove the ``if'' part.
Assume, therefore, that $\overline{\operatorname{span}}\,\{ab : a,b\in A, a\circ b =0\}$ contains all commutators and
that a continuous bilinear functional $\varphi \colon A\times A\to\mathbb{C}$ satisfies
\begin{equation*}
a,b\in A, \ a\circ b=0 \ \Longrightarrow \ \varphi(a,b)=0.
\end{equation*}
By Lemma \ref{1401},  there exist $\sigma,\tau\in A^*$   such that
\[
\varphi(a,b)= \sigma(a\circ b)+\tau([a,b])\quad (a,b\in A).
\]
 Obviously, $\tau(ab)=\frac{1}{2}\tau([a,b])=0$ whenever $a\circ b=0$. However, according to our assumption this implies that
$\tau$ vanishes on all commutators. Hence, $\varphi(a,b)= \sigma(a\circ b)$, which proves that $A$ is zero Jordan product determined.
\end{proof}

The following is an immediate consequence of Theorem \ref{pr2}.

\begin{corollary}\label{co2}
Let $A$ be a zero Jordan product determined Banach algebra.
If $A$ is not commutative, then there exist $a,b\in A$ such that $ab\ne 0$ and $a\circ b=0$.
\end{corollary}

Is the closure unnecessary in  Theorem \ref{pr2}?
A rather obvious way to attack this question is to consider the purely algebraic
version of the zero Jordan product determination.
We say that an algebra $A$ over a field $\mathbb{F}$ is a \emph{zero product determined algebra}
if every bilinear map $\varphi$ from $A\times A$ into each linear space $X$ that satisfies \eqref{P}
can be written in the standard form \eqref{PS} for some linear map $\sigma\colon A^2\to X$.
We  say that an algebra $A$ over a field $\mathbb{F}$ is a \emph{zero Jordan product determined algebra}
if every bilinear map $\varphi$ from $A\times A$ into each linear space $X$ that satisfies \eqref{J}
can be written in the standard form \eqref{S} for some linear map $\sigma\colon A\circ A\to X$. These notions were introduced in 
\cite{BGS}. Concerning the second notion, the most general result we know 
 of is the following.

\begin{theorem}\label{idem} \cite{ALH}
A unital algebra over a field of characteristic not $2$ is zero Jordan product determined if it is generated by idempotents.
\end{theorem}

It is clear that a commutative algebra over a field of characteristic not $2$ is a zero product determined algebra
if and only if it is a zero Jordan product determined algebra.
Therefore, there do exist commutative algebras that are not zero Jordan product determined.
What about  noncommutative algebras? We first examine two important examples of finite-dimensional algebras.

\begin{example}\label{ex1}
The real algebra of quaternions $\mathbb H$ is not a zero Jordan product determined algebra.
To show this,  consider $\mathbb H$ as the real vector space $\mathbb R\times \mathbb R^3$ given by the product
\[
(\alpha,x)(\beta,y)= (\alpha\beta - x\cdot y,\alpha y + \beta x  + x\times y)
\]
(here, $x\cdot y$ and  $x\times y$ stand for the  dot product and   the vector product of $x$ and $y$, respectively).
Assuming that
\[
(\alpha,x)(\beta,y) + (\beta,y)(\alpha,x) =0,
\]
it follows that $\alpha\beta = x\cdot y$ and $\alpha y =- \beta x$, and hence
\[
\alpha^2\beta^2= \beta x\cdot\alpha y= - \beta x\cdot\beta x,
\]
which is possible only if both sides are zero. In particular, $\alpha\beta=0$. This means that the bilinear map
$\varphi\colon\mathbb H\times \mathbb H\to \mathbb R$ given by
\[
\varphi\big((\alpha,x),   (\beta,y)\big) =\alpha\beta
\]
satisfies the condition \eqref{J}.
Since
\[
\varphi\big((1,0),(1,0)\big) + \varphi\big((0,i), (0,i)\big) = 1\cdot 1 + 0\cdot 0 =1
\]
and
\[
(1,0)(1,0)  + (0,i)(0,i) =0,
\]
we see that there does not exist a linear functional $\sigma$ on $\mathbb H$ such that \eqref{S} holds.
\end{example}

\begin{example}\label{ex2}
Let $A$ be the Grassmann algebra (over, say, $\mathbb C$) in two generators $x$ and $y$.
That is, $A$ is the $4$-dimensional linear space with basis $1,x,y,xy$ whose multiplication
is determined by $x^2 =y^2 =x\circ y=0$. Write
\[
a= \lambda_0 + \lambda_1 x + \lambda_2 y + \lambda_3 xy\quad\mbox{and}\quad b=\mu_0 + \mu_1 x + \mu_2 y + \mu_3 xy.
\]
It is easy to see that $a\circ b=0$ if and only if either $a=0$, $b=0$, or
$\lambda_0=\mu_0=0$.
Hence, the bilinear functional given by
\[
\varphi(\lambda_0 + \lambda_1 x + \lambda_2 y + \lambda_3 xy, \mu_0 + \mu_1 x + \mu_2 y + \mu_3 xy) = \lambda_0\mu_1
\]
satisfies the condition \eqref{J}. However, $\varphi$ is not symmetric, specifically
$\varphi(1,x)$ is not equal to $\varphi(x,1)$, and therefore it cannot be written in the form \eqref{S}.
Therefore, $A$ is not a zero Jordan product determined algebra.
\end{example}

An obvious modification of the proof of Theorem \ref{pr2} gives the following.

\begin{theorem}\label{pr}
Let $A$ be a zero Jordan product determined  algebra. Then every commutator in $A$ lies in
$\operatorname{span}\,\{ab : a,b\in A, a\circ b =0\}$.
\end{theorem}

Of course, Corollary~\ref{co2} also holds for zero Jordan product determined  algebras.

\begin{corollary}
Let $A$ be a zero Jordan product determined algebra.
If $A$ is not commutative, then there exist $a,b\in A$ such that $ab\ne 0$ and $a\circ b=0$.
\end{corollary}

The following example shows that there are algebras not having such a pair of elements.

\begin{example}
Let $A$ be the (first) Weyl algebra over $\mathbb C$. Recall that $A$ is generated by two elements $x$ and $y$
that satisfy the relation $xy-yx=1$.
Every element in $A$ can be uniquely written in the form
$\sum_{i=0}^m f_i(x)y^i$ for some $n\ge 0$ and some polynomials $f_i$, and  for any polynomials $f_i,g_j$ there exists polynomials $h_k$ such that
\[
\left(\sum_{i=0}^m f_i(x)y^i\right)\left(\sum_{j=0}^n g_j(x)y^j\right) = f_m(x)g_n(x)y^{m+n} + \sum_{k=0}^{m+n-1} h_k(x)y^k
\]
(see \cite[Example 2.28]{B}). This readily implies that $a\circ b\ne 0$ for all nonzero elements $a,b\in A$.
\end{example}

Thus, neither the algebra of quaternions nor the Weyl algebra is a zero Jordan product determined algebra.
We remark that these two algebras have no zero-divisors,
so they are not zero product determined algebras.
The Grassmann algebra also is not a zero product determined algebra since it is finite-dimensional and is
not generated by idempotents \cite{B1}.
As a matter of fact, we do not know of any unital algebra that is either zero Jordan product determined or zero product determined,
but is not generated by idempotents.

The converse of Theorem \ref{pr} does not hold. That is, if every commutator lies in
$\operatorname{span}\,\{ab : a,b\in A, a\circ b =0\}$, then
$A$ may not be zero Jordan product determined. Both algebras from Examples \ref{ex1} and \ref{ex2} serve as counterexamples.

Combining Theorem \ref{pr} with Theorem \ref{idem}, we thus see that in a unital
algebra that is generated by idempotents, every commutator can be written as a sum of elements of the form $ab$ where $a$ and $b$ anticommute.
In particular, this holds for every matrix algebra $M_n(B)$, where $n\ge 2$ and $B$ is any unital algebra. However, more can be said about commutators
in an algebra of this kind. The following lemma
is evident from the proof of \cite[Theorem 4.4]{ABESV}.

\begin{lemma}\label{lemnb}
Let $B$ be a unital algebra and  let $n\ge 2$. Then every commutator in $A=M_n(B)$ can be written as a sum of $22$ elements of the form
$ex(1-e)$ where $e,x\in A$ with $e$ being an idempotent.
\end{lemma}

Since $ex(1-e) = ex(1-e)\cdot (1-2e)$ and  $ ex(1-e)$ and $1-2e$ anticommute, this yields the following.

\begin{proposition}
Let $B$ be a unital algebra and  let $n\ge 2$. Then every commutator in $A=M_n(B)$ is a sum of
$22$ elements of the form $ab$ where $a,b\in A$ and $a\circ b =0$.
\end{proposition}

Using \cite[Proposition 4.5]{ABESV} along with a brief inspection
of the proofs of \cite[Theorem~4.4 and 4.6]{ABESV}, shows that
Lemma~\ref{lemnb} holds for every von Neumann algebra.
Therefore, the following theorem is true.

\begin{theorem}
Let $A$ be a von Neumann algebra. Then every commutator in $A$ is a sum of
$22$ elements of the form $ab$ where $a,b\in A$ and $a\circ b =0$.
\end{theorem}

\bibliographystyle{line}

\end{document}